\DeclareSymbolFont{AMSb}{U}{msb}{m}{n}
\documentclass[11pt,a4paper,twoside,leqno,noamsfonts]{amsart}
\usepackage[utopia]{mathdesign}
\usepackage{microtype}

\usepackage[english]{babel}
\usepackage{amsmath, amsthm, amstext}
\usepackage[all,cmtip]{xy}
\usepackage[T1]{fontenc}
\usepackage{xcolor}
\usepackage{multicol}

\usepackage[all,cmtip]{xy}	
\xyoption{arrow}

\theoremstyle{plain}
\newtheorem{thm}{Theorem}[section]
\newtheorem*{thm*}{Theorem}
\newtheorem{lem}[thm]{Lemma}
\newtheorem{prp}[thm]{Proposition}
\newtheorem{cor}[thm]{Corollary}
\newtheorem*{cor*}{Corollary}

\theoremstyle{definition}
\newtheorem{defn}[thm]{Definition}

\newtheorem{rem}[thm]{Remark}
\newtheorem{exmp}[thm]{Example}

\newcommand{\bC}{\mathbb{C}}

\newcommand{\bP}{\mathbb{P}}

\newcommand{\bZ}{\mathbb{Z}}

\newcommand{\cD}{\mathcal{D}}

\newcommand{\cF}{\mathcal{F}}

\newcommand{\cO}{\mathcal{O}}

\newcommand{\sA}{\mathscr{A}}
\newcommand{\sB}{\mathscr{B}}
\newcommand{\sC}{\mathscr{C}}

\newcommand{\sE}{\mathscr{E}}
\newcommand{\sF}{\mathscr{F}}
\newcommand{\sG}{\mathscr{G}}

\newcommand{\sK}{\mathscr{K}}
\newcommand{\sL}{\mathscr{L}}
\newcommand{\sO}{\mathscr{O}}

\newcommand{\sV}{\mathscr{V}}
\newcommand{\sW}{\mathscr{W}}

\newcommand{\perf}{\operatorname{Perf}}
\newcommand{\coker}{\operatorname{Coker}}
\renewcommand{\ker}{\operatorname{Ker}}
\newcommand{\im}{\operatorname{Im}}
\newcommand{\ext}{\operatorname{Ext}}

\newcommand{\coh}{\operatorname{Coh}}
\newcommand{\aut}{\operatorname{Aut}}
\newcommand{\id}{{\operatorname{Id}}}
\newcommand{\tor}{{\operatorname{Tor}}}

\newcommand{\res}{{\operatorname{res}}}
\renewcommand{\inf}{{\operatorname{ind}}}
\newcommand{\triv}{{\operatorname{triv}}}

\title[Descent criterion for equivariant derived categories]{A descent criterion for equivalences between equivariant derived categories}

\author{Francesco Amodeo}
\email{francescoamodeo87@gmail.com}

\author{Riccardo Moschetti}
\address[R.M.]{ Dipartimento di Matematica ``F. Casorati'', Universit\`{a} degli Studi di Pavia, Via Ferrata 5, 27100 Pavia, Italy}
\email{riccardo.moschetti@gmail.com}

\author{Mattia Ornaghi}
\address[M.O.]{Erwin Schr\"odinger International Institute for Mathematics and Physics,
Boltzmanngasse 9, 1090 Wien, Austria}
\email{mattia12.ornaghi@gmail.com}

\keywords{Derived categories, Derived categories of equivariant sheaves, Derived category of quotient variety.}
\subjclass[2010]{13D09, 18E30, 14F43}

\begin{document}
\begin{abstract}
We investigate equivalences between the categories of perfects complexes of the quotients of two smooth projective variety by the action of a finite group.
As a result we give a necessary and sufficient condition for an equivalence between the equivariant derived categories to descend to the categories of perfect complexes.
\end{abstract}

\maketitle


\section{Introduction} 
Let $G$ be a finite group acting on a smooth projective variety $X$.\ We can consider the category $\coh^G(X)$ of $G$-equivariant coherent sheaves on $X$ and the corresponding bounded derived category, which we will denote by $\cD^G(X)$.\ 
When the action is free, the quotient $X/G$ is smooth and $\cD^G(X)$ turns out to be equivalent to the bounded derived category of coherent sheaves of the quotient $\cD(X/G)$.\ This is no longer true when the action is not free. However, $\cD^G(X)$ is still equivalent to the derived category of the stack quotient $[X/G]$. Notice that $[X/G]$ is smooth as a stack, so we can think of $\cD^G(X)$ as a replacement for the bounded derived category of $X/G$ when the action is not free. The autoequivalences of the aformentioned categories were studied in \cite{Ploog}. This is a way to understand the relations between $\cD(X/G)$ and $\cD^G(X)$ for a non-free action, as a categorical analogous of the comparison between the quotient $X/G$ and the stack quotient $[X/G]$. 

Another replacement for the category $\cD(X/G)$ when $X/G$ is singular is provided by the subcategory $\perf(X/G)$. It is called the category of perfect complexes and consists of the objects in $\cD(X/G)$ which are quasi-isomorphic to bounded complexes of locally free sheaves of finite rank on $X/G$. We refer to \cite{KPS} for more details about perfect complexes and equivariant derived categories. The aim of this paper is to compare $\cD^G(X)$ and $\perf(X/G)$ by studying when equivalences between equivariant derived categories descend to equivalences of the categories of perfect complexes.

In order to state our main result, consider another finite group $H$ acting on a smooth projective variety $Y$. Orlov's representability theorem for smooth stacks was proven by Kawamata, see \cite[Theorem 1.1]{Kaw}. As an outcome, we have that every fully faithful exact functor $F$ from $\cD^G(X)$ to $\cD^H(Y)$ is of Fourier-Mukai type, meaning that there exists an object $\sE^{\bullet}$ in $\cD^{G \times H}(X \times Y)$, such that $F$ is isomorphic to the functor $\Phi_{\sE^{\bullet}}$ given by $q_{Y,*}^G(q_X^*(-) \otimes \sE^\bullet)$, where $q_X$ and $q_Y$ are the projections to the two factors of $X \times Y$ and all the functors are properly derived when needed. The functor $q_{Y,*}^G$ is the equivariant pushforward described in Section \ref{Sec:dceq}. The functor $\Phi_{\sE^{\bullet}}$ is called Fourier-Mukai functor with kernel $\sE^{\bullet}$. When $\Phi_{\sE^{\bullet}}$ is an equivalence, an inverse is provided by its left adjoint $\Phi_{\sE_L^{\bullet}}$ (or equivalentely by its right adjoint).

In order to get an equivalence between $\perf(X/G)$ and $\perf(Y/H)$ starting from $\Phi_{\sE^{\bullet}}$, we will define two functors to compare equivariant and non equivariant settings. The projection $\pi:X\to X/G$ is compatible with the action of $G$ on $X$ and with the trivial action of $G$ on $X/G$. As a consequence, we get a functor $\triv: \coh(X/G) \to \coh^G(X/G)$ which equips a sheaf with the trivial action, and its adjoint $[\cdot]^G:\coh^G(X/G) \to \coh(X/G)$ which takes the $G$-equivariant part of a sheaf with a linearisation (see Section \ref{Sec:dceq} for the details). The derived pushforward of $\pi$ (resp. pullback) can be composed with the derived $[\cdot]^G$ (resp. $\triv$) to get a new functor $\Pi^G_*$ (resp. $\Pi^{G,*}$), as in the follwing diagram:

$$
\xymatrix{
\cD(X/G) \ar@<.5ex>[r]^-{\triv}   & \cD^G(X/G)\ar@<.5ex>[r]^-{\pi^*}\ar@<.5ex>[l]^-{[\cdot]^G} & \cD^G(X).\ar@<.5ex>[l]^-{\pi_{*}}\\
}
$$
This construction can be performed also on $Y$ with the action of $H$; we collect all of this together in the following diagram by considering also a Fourier-Mukai equivalence $\Phi_{\sE^{\bullet}}$:
$$
\xymatrix{
\cD^G(X)\ar[r]^{\Phi_{\sE^{\bullet}}} & \cD^H(Y)\ar[d]^{\Pi^{H}_*}\\
\cD(X/G)\ar@<.5ex>[u]^{\Pi^{G,*}} & \cD(Y/H)\\
\perf(X/G)\ar@{^{(}->}[u] & \perf(Y/H)\ar@{^{(}->}[u]
}
$$

We call $\Omega$ the functor obtained by the composition $\Pi^{H}_{*} \circ \Phi_{\sE^{\bullet}} \circ \Pi^{G,*}$ restricted to $\perf(X/G)$. This is a functor from $\perf(X/G)$ to $\cD(Y/H)$. The aim of this paper is to understand if an equivalence $\Phi_{\sE^\bullet}$ induces an equivalence between $\perf(X/G)$ and $\perf(Y/H)$ through $\Omega$, meaning that $\Omega$ has image in the category $\perf(Y/H)$ and is an equivalence. The key is a property called \textit{descent}: we say a $G$-equivariant complex of sheaves $\sF^\bullet$ descends to $\perf(X/G)$ if there exists a complex $\sV^\bullet$ in $\perf(X/G)$ such that $\Pi^{G,*}(\sV^\bullet)$ is isomorphic to $\sF^\bullet$. We make the role of this property explicit in Theorem \ref{teoremadellasvolta}:

\begin{thm*} 
The functor{\ } $\Omega$ is an equivalence between $\perf(X/G)$ and $\perf(Y/H)$ if and only if the two following conditions are satisfied.
\begin{itemize}
\item[(i)] The object $\Phi_{\sE^\bullet} \circ \Pi^{G,*}(\sF^{\bullet})$ descends to $\perf(Y/H)$, for all $\sF^{\bullet} \in \perf(X/G)$.
\item[(ii)] The object $(\Phi_{\sE^\bullet})^{-1} \circ \Pi^{H,*}(\sV^{\bullet})$ descends to $\perf(X/G)$, for all $\sV^{\bullet}\in \perf(Y/H)$.
\end{itemize}
\end{thm*}

We then find an explicit condition which is equivalent to the descent property required by the previous theorem. This is done in Theorem \ref{PG}: 

\begin{thm*} 
An object $\sF^\bullet$ in $\cD^G(X)$ descends to $\perf(X/G)$ if and only if the following condition holds.

\begin{itemize}
\item[{$(\star)$}] The stabiliser $G_x$ acts trivially on the $\sO_X$-modules $H^j(\sF^\bullet \otimes k(x))$, for all $x \in X, j \in \bZ$.
\end{itemize}
\end{thm*}

Combining the two results, meaning checking $(\star)$ for every complex of Conditions (i) and (ii) gives a condition which is difficult to study, since it has to be verified on every perfect complex involved. We point out that the problem of checking if $\Omega$ actually defines an equivalence $\perf(X/G) \to \perf(Y/H)$ is not trivial, even when we restrict to autoequivalences. In order to clarify this, we give an explicit example at the end of Section \ref{sec:statementmainfunctors}. Such an example suggests that the descent criterion of Theorem \ref{teoremadellasvolta} may only depend on the kernel $\sE^\bullet$ of the Fourier-Mukai equivalence we are considering. In this perspective, we exhibit in Corollary \ref{cornec} a necessary condition depending only on $\sE^\bullet$:

\begin{cor*}
The functor $\Omega$ is an equivalence between $\perf(X/G)$ and $\perf(Y/H)$ only if the stabiliser $H_y$ acts trivially on $\left[H^i\bigl(\sE^\bullet|_{X \times \{y\}} \bigr)\right]^G$ and the stabiliser $G_x$ acts trivially on $\left[H^i\bigl(\sE_L^\bullet|_{\{x\} \times Y} \bigr)\right]^H$ for every integer $i$, $x$ in $X$ and $y$ in $Y$.
\end{cor*}

Unfortunately the corollary is only a necessary condition. Our main result is Theorem \ref{thm:final2} which gives conditions which has to be checked only on a generator of $\perf(X/G)$ and a generator of $\perf(Y/H)$.

\begin{thm*} 
Consider a generator $\sA^\bullet$ of $\perf(X/G)$ and a generator $\sB^\bullet$ of $\perf(Y/H)$. The functor { }$\Omega$ is an equivalence between $\perf(X/G)$ and $\perf(Y/H)$ if and only if the two following conditions are satisfied for every integer $i$, $y$ in $Y$, $x$ in $X$.
\begin{itemize}
\item[(i')] The stabiliser $H_y$ acts trivially on $\left[H^i\bigl(\Pi^{*,G} (\sA^\bullet) \otimes \sE^\bullet|_{X \times \{y\}} \bigr)\right]^G$.

\item[(ii')] The stabiliser $G_x$ acts trivially on $\left[H^i\bigl(\Pi^{*,H} (\sB^\bullet) \otimes \sE_L^\bullet|_{\{x\} \times Y} \bigr)\right]^H$.

\end{itemize}
\end{thm*}

\subsection*{Plan of the paper.}
The paper is organised as follows: in Section \ref{sec:background} we give some basic notions on derived categories and equivariant sheaves. The functor $\Omega$ is central in this paper and it is defined in Section \ref{sec:statementmainfunctors}, where we also prove Theorem \ref{teoremadellasvolta}. At the end of Section \ref{sec:statementmainfunctors} there are examples describing the possible behaviours of $\Omega$. The problem of finding when $\Omega$ is an equivalence is related to the descent data of sheaves: we give an explicit condition for this in Theorem \ref{PG}. Finally, Section \ref{sec:applications} we apply all the previous result to get the more explicit criterion Theorem \ref{thm:final2}.

\subsection*{Conventions}
We will consider functors properly derived where needed. We will use $\sF$ without decorations to identify a sheaf and $\sG^\bullet$ to identify a complex of sheaves. Throughout the paper, we work over a base field $k$ which is algebraically closed of characteristic zero. Even if some concept can be stated also more generally, we will always work with finite groups.

\section{Preliminaries} \label{sec:background}

\subsection{Group actions}

\begin{defn} \label{defn:groupaction}
A \textit{(left) action} of $G$ on a variety $X$ is given by a group homomorphism from $G$ to $\aut(X)$. We will call this map with the same name $g$.
\end{defn}

We refer to the book \cite{isaacs} for all the details on actions by finite groups; we will report here only the details that we need in this paper.

\begin{defn}
Let $U$ be a subset of $X$. The \textit{orbit} of $U$ is the subset of $X$ defined as
$$G(U):=\left\{g(x)\in X\mbox{ s.t. } x \in U, g \in G\right\}.$$
The \textit{stabiliser} $G_x$ of the point $x \in X$ is the subset of $G$ defined as
$$G_x:=\left\{g\in G\mbox{ s.t. } g(x)=x\right\}.$$
\end{defn}

The action of a group is called \textit{free}, if every point has a trivial stabiliser.

\begin{exmp} \label{exa:action1}
Consider the action of the group $\bZ_2$ on $\bP^2(x_0:x_1:x_2)$ given by
\begin{equation*}
{g}(x_0:x_1:x_2):=
\begin{cases}
(x_0:x_1:x_2) &\mbox{if } g=0,\\
(x_0:x_1:-x_2) &\mbox{if } g=1.
\end{cases}
\end{equation*}
It is straightforward to check that the orbits of points are given by 
\begin{equation*}
{\mathbb{Z}_2}(x_0:x_1:x_2):=
\begin{cases}
(x_0:x_1:x_2), (x_0:x_1:-x_2) &\mbox{if $x_2\neq 0$},\\
(x_0:x_1:0) &\mbox{if $x_2=0$}.
\end{cases}
\end{equation*}
Moreover, the stabilisers of the point $(x_0:x_1:0)$ and $(0:0:1)$ is the whole $\bZ_2$. It is the identity group for every other point. Finally, the fixed locus for the action of $\bZ_2$ is given by
$$\big\{[z_0:z_1:z_2] \in \bP^2 \text{ such that } z_2=0 \big\} \cup \big\{[0:0:1]\big\}.$$
\end{exmp}

Once we have an action of $G$ on $X$ as described before, we can certainly consider the topological quotient $X/G$, where the points are actually $G$-orbits. We want to understand when $X/G$ can be endowed with a scheme structure induced by $X$ via the projection map.

\begin{defn}
The \textit{geometric quotient} of $X$ by $G$ is a pair ($X/G,\pi$), where $X/G$ is a projective variety and $\pi$ is a map $X\to X/G$, such that:
\begin{itemize}
\item[1.] The set $\pi^{-1}(y)$ coincides with the orbit $G(\{y\})$ for every $y\in X/G$;
\item[2.] The subset $U$ is open in $X/G$ if and only if $\pi^{-1}(U)$ is open in $X$;
\item[3.] The structure sheaf $\cO_{X/G}$ is isomorphic to $\pi_*(\cO_X^G)$, where $\cO_X^G$ denotes the $G$-linearised structure sheaf (cfr. Section \ref{Sec:dceq}).
\end{itemize}
\end{defn}

The geometric quotient ($X/G,\pi$) is a \textit{categorical quotient}; it means that the following conditions are satisfied:
\begin{itemize}
\item[1.] The map $\pi$ is $G$-invariant, i.e. for every $g\in G$ we have $\pi\circ g=\pi$.
\item[2.] The map $\pi$ is universal, i.e. for every pair ($Z,\pi'$), where $Z$ is a projective variety and $\pi':X\to Z$ is $G$-invariant, there exists a unique map $h:X/G\to Z$ such that the following diagram commutes:
$$\xymatrix{
X\ar[r]^{\pi}\ar[dr]_{\pi'}&X/G\ar@{-->}[d]^{\exists! h}\\
&Z.
}$$
\end{itemize}
When the group $G$ is finite, a geometric quotient always exists. In the notation above, it is given by the topological quotient $\pi: X \to X/G$. When $G$ is infinite, the situation is more complicated and a possible solution comes from the Geometric Invariant Theory.
Starting from a smooth variety $X$, it is easy to prove that the quotient $X/G$ is smooth if and only if the action of $G$ is free.

\begin{exmp} \label{exa:action2}
Let's go back on Example \ref{exa:action1} where the action of $\bZ_2$ on $\bP^2$ is not free. The quotient is isomorphic to the weigthed projective plane $\bP^2(1:1:2)$. This is isomorphic to the quadric cone defined in $\bP^3$ by
$$\left\{[w_0:w_1:w_2:w_3] \in \bP^3 \text{ such that } w_0w_1=w_2^2\right\}.$$
This is a surface with a singular point in $[0:0:0:1]$; the isomorphism with the weigthed projective plane is given explicitly by
$$w_0=z_0^2, w_1=z_0z_1, w_2=z_2^2, w_3=z_2^2.$$
\end{exmp}

\subsection{Derived category of $G$-equivariant sheaves} \label{Sec:dceq}
We will closely follow \cite{PloogThesis} and the related paper \cite{Ploog}, which contains a very extensive introduction to the topic.
As in Definition \ref{defn:groupaction}, every element $g \in G$ provides a map from $X$ to $X$. The pushforward $g_*$ and pullback $g^*$ of sheaves along the morphism $g$ make sense. We will consider the category of $G$-linearised coherent sheaves: it can be defined for any algebraic group, see \cite[Definition 3.1]{PloogThesis} and more in general \cite{BeLu}. When dealing with finite group, it is possible to work with this definition, equivalent to the general one.

\begin{defn}
Let $\sF$ be a coherent sheaf of $X$. A \textit{$G$-linearisation} of $\sF$ is a family of isomorphisms $\lambda:=\{\lambda_g:\sF \xrightarrow{\sim} g^*\sF\}_{g\in G}$ such that: $\lambda_{\id}=\id_\sF$ and $\lambda_{gh}=h^*\lambda_g \circ \lambda_h$ for every $g,h\in G$.
\end{defn}

The category $\coh^G(X)$ of $G$-equivariant coherent sheaves on $X$ consists of objects $(\sF,\lambda)$ where $\sF$ belongs to $\coh(X)$ and $\lambda$ is a $G$-linearisation of $\sF$. The morphisms  are given by $f:(\sF,\lambda)\to(\sF',\lambda')$ such that the following diagram commutes for every $g \in G$:
\[
\xymatrix{
\sF\ar[r]^-{\lambda_{g}}\ar[d]_{f}   & g^*\sF\ar[d]^{g^*f}\\
\sF'\ar[r]_-{\lambda'_g}              & g^*\sF'.
}
\]

Notice that the structure sheaf $\sO_X$ admits a canonical $G$-linearisation given by the push forward of regular functions $\sO_X \to g^* \sO_X$ for any $g \in G$.
The category $\coh^G(X)$ is abelian, so we can consider the bounded derived category  $\cD^G(X):=\cD(\coh^G(X))$ associated to it.

Assume now that the action of $G$ on $X$ is trivial. In this case, a $G$ linearisation of a coherent sheaf $\sF$ on $X$ is just given by a group homomorphism $\lambda: G \to \aut(\sF)$, that is also called representation of $G$. Recall that given a $\lambda$ invariant subsheaf $\sF_i \subset \sF$, we can restrict $\lambda$ to $\sF_i$ obtaining $\lambda_i: G \to \aut(\sF_i)$, called a subrepresentation of $G$. A representation is called irreducible if it does not admit non-trivial subrepresentations. We can consider the irreducible subrepresentations of $\lambda$, writing $(\sF,\lambda)$ as a direct sum of $(\sF_i,\lambda_i)$ in $\coh^G(X)$. We are interested in particular to the one corresponding to the identity automorphism of $\sF$, we will call it $(\sF_0,\lambda_0)$, where $\lambda_0: G \to \aut(\sF_0)$ is called the trivial action and sends every $g \in G$ to $\id_{\sF_0}$. The subsheaf $\sF_0$ is also called \textit{$G$-invariant part} of $\sF$. We can define the functor which equips each coherent sheaf with the trivial action as
\begin{align} 
\begin{split}
\triv: \coh(X) &\to \coh^G(X)\\
\sF &\mapsto (\sF,\lambda_0).
\end{split}
\end{align}
This functor has an adjoint which takes the $G$-invariant part:
\begin{align} 
\begin{split}
[\cdot]^G: \coh^G(X) &\to \coh(X)\\
(\sF,\lambda) &\mapsto \sF_0.
\end{split}
\end{align}
Notice that if the action of $G$ is not necessarely trivial, but there exist a normal subgroup $H$ acting trivially, we can construct $[\cdot]^H$ makes sense, and it goes from $\coh^G(X)$ to $\coh^{G/H}(X)$. Consider now two groups $G$ and $H$ acting on two smooth projective varieties $X$ and $Y$. Assume we have a homomorphism $\phi:G \to H$.

\begin{defn}
A \textit{$\phi$-map} between $X$ and $Y$ is a morphism $\alpha:X \to Y$ such that the following diagram commutes for every $g \in G$:
\[
\xymatrix{
X\ar[d]^g\ar[r]^\alpha & Y\ar[d]^{\phi(g)}\\
X\ar[r]^\alpha & Y.
}
\]
\end{defn}
This allows us to get a well defined map from $X/G$ to $Y/H$. We can define the \textit{equivariant pull back} of a $\phi$-map $\alpha$ as follows:
\begin{align*}
    \alpha^*:\coh^H(Y) &\to \coh^G(X)\\
    (\sF,\lambda) &\mapsto (\alpha^*\sF, \alpha^*\lambda),
\end{align*}
where $\alpha^*\lambda(g)$ sends $\alpha^*\sF$ to $g^*\alpha^*\sF$ for every $g \in G$. Assume $\phi$ surjective. In order to define the equivariant pushforward of $\alpha$ we must proceed in two steps. First recall that if $G=H$, the pushforward $\alpha_* \sF$ has a natural linearisation for every $\sF \in \coh^G(X)$. 
Notice that the group $G$ acts also on $Y$ thanks to the map $\phi$, so we have a well-defined pushforward from $\coh^G(X)$ to $\coh^G(Y)$. Second, notice that the kernel $K:= \ker(\phi)$ acts trivially on $Y$, so we can apply the functor $[\cdot]^K:\coh^G(Y) \to \coh^H(Y)$. The \textit{equivariant pushforward} is defined as the composition of these two functors:
$$\alpha^K_*:\coh^G(X) \xrightarrow{\alpha_*} \coh^G(Y) \xrightarrow{[\cdot]^K} \coh^H(Y).$$

\subsection{Fourier-Mukai transform} 

Consider two algebraic spaces $X$, $Y$.\ A functor $F$ from $\cD(X)$ to $\cD(Y)$ is called of Fourier-Mukai type if there exists an object $\sE^\bullet$ in $\cD^b(X\times Y)$ such that $F$ is isomorphic to the functor defined by
\begin{align*}
\Phi_{\sE^{\bullet}}:\cD(X) &\to\cD(Y)\\
                \cF^\bullet &\mapsto {q}_{Y,*}(q_X^*\cF^\bullet \otimes\sE^\bullet),
\end{align*}
where $q_X$ and $q_Y$ denote the projections from $X \times Y$ to $X$ and $Y$, respectively. Fourier-Mukai functors have a deep role in algebraic geometry, they are exact and well-behaved with respect to composition, and there is an active and rich field of study about finding conditions to describe which functors are actually of Fourier-Mukai type, see \cite{CaSeSurveyFM} for a survey on this topic.
When dealing with equivariant derived categories we have to slightly change the definition of Fourier-Mukai functors to take into account the linearised setting. Consider two finite groups $G$ and $H$ acting on two smooth varieties $X$ and $Y$, respectively. If we take a kernel $\sE^\bullet$ in $\cD^{G \times H}(X\times Y)$ we define its corresponding Fourier-Mukai transfom as
\begin{align*}
\Phi_{\sE^{\bullet}}:\cD^G(X) &\to\cD^H(Y)\\
                \cF^\bullet &\mapsto {q}_{Y,*}^G(q_X^*\cF^\bullet \otimes\sE^\bullet),
\end{align*}
where we are taking the equivariant pushforward along the map $q_Y$, since $G$ is the kernel of the projection $G \times H \to H$.

We can follow \cite{KS} in order to obtain that the derived category $\cD([X/G])$ of the stack $[X/G]$ is equivalent to the equivariant derived category $\cD^G(X)$, the same holds with $Y/H$. See also \cite[Remark 3.14]{PloogThesis}. It is then a consequence of \cite[Theorem 1.1]{Kaw} that all the exact equivalences between $\cD^G(X)$ and $\cD^G(Y)$ are of Fourier-Mukai type; in particular, the kernel $\sE^{\bullet}$ lives in $\cD^{G \times H}(X \times Y)$. 

Assume that the functor $\Phi_{\sE^\bullet}$ is an equivalence. It has a left adjoint which is a Fourier-Mukai functor with kernel $\sE_L^\bullet := \sE^{\bullet, \vee} \otimes q_Y^* \omega_Y[\dim(Y)]$, and a right adjoint given by $\sE_R^\bullet := \sE^{\bullet, \vee} \otimes q_X^* \omega_X[\dim(X)]$. The objects $\sE_L^\bullet$ and $\sE_R^\bullet$ are isomorphic and both $\Phi_{\sE_L^\bullet}$ and $\Phi_{\sE_R^\bullet}$ are inverse to $\Phi_{\sE^\bullet}$. See \cite[Section 5]{HuyFM} for details.

\section{Descent data and equivalences of perfect complexes} \label{sec:statementmainfunctors}
In all this section we will always work in the setting of two finite groups $G$ and $H$ acting on two smooth varieties $X$ and $Y$, respectively.

Thanks to the result of Kawamata all the exact equivalences between $\cD^G(X)$ and $\cD^H(Y)$ are of Fourier-Mukai type, so we will fix a kernel $\sE^{\bullet}$ in $\cD^{G\times H}(X\times Y)$ such that the corresponding Fourier-Mukai functor $\Phi_{\sE^{\bullet}}$ is an equivalence.
We can take the trivial action of $G$ on the quotient $X/G$ and consider the equivariant derived category $\cD^G(X/G)$, similar for $H$ and $Y/H$. Thanks to the results of the previous section in the context of coherent sheaves, we can consider the induced functors at the level of derived categories to get the following diagram:

\[
\xymatrix{
\cD^G(X)\ar[rr]^{\Phi_{\sE^{\bullet}}}  && \cD^H(Y)   \ar[d]_{\pi_{*}}\ar@/^3pc/[dd]^{\Pi^{H}_{*}}\\
\cD^G(X/G)\ar[u]_{\pi^*}         && \cD^H(Y/H)   \ar[d]_{[\cdot]^H}\\
\ar@/^3pc/[uu]^{\Pi^{G,*}}    \cD(X/G)\ar[u]_{\triv}             && \cD(Y/H) \\
\perf(X/G)\ar@{-->}[rr]\ar@{^{(}->}[u] \ar[urr]^{\Omega}     && \perf(Y/H) \ar@{^{(}->}[u]\\
}
\]
We wrote explicitly only the maps we will use, but obviously all the construction can be done symmetrically for both $(X,G)$ and $(Y,H)$.
We define the functor $\Pi^{G,*}$ as the composition $\pi^{*} \circ \triv$ and $\Pi^H_{*}$ as the composition $[\cdot]^H \circ \pi_{*}$. We will use the same name for the corresponding functors between the category of sheaves, and we will not specify of we are referring to the $X$ or the $Y$ side, since that will be always clear from the context. The functor $\Omega$ is defined to be the composition $\Pi^{H}_{*} \circ \Phi_{\sE} \circ \Pi^{G,*}$, restricted to the subcategory of perfect complexes of $X/G$.

\begin{rem}
There is another way to compare the equivariant and the non equivariant setting. Even when the action of $G$ is not free, we can always consider the restriction functor $\res: \coh^G(X) \to \coh(X)$ which just forget the linearisation. The functor $\res$ has an adjoint called the induction functor:
\begin{align*}
\inf:\mbox{Coh}(X)& \to \mbox{Coh}^G(X)\\
\sF &\mapsto (\bigoplus_{g\in G}g^*\sF,\rho_g),
\end{align*}
where $\rho_g$ is the linearisation which comes from the permutations of the summands. Both $\inf$ and $\res$, as well as $\triv$ and $[\cdot]^G$ commutes with pullback and pushforward. We refer to Section 2.2 of \cite{Krug} for the definition of these functors and their properties. These functor are not good for our puropose: we aim to find an equivalence, but the direct sum which defines the functor $\inf$ makes sure that there is no way to use this functor to get an equivalence, because it will send indecomposable objects to decomposable one.
\end{rem}

We want to find conditions to guarantee that the functor $\Omega$ is an equivalence, with image in the category $\perf(Y/H)$. 

\begin{defn} \label{def:descendcoherent}
A $G$-equivariant locally free sheaf of finite type $\sF$ \emph{descends} to $\coh(X/G)$ if there exists a locally free sheaf $\sV$ in $\coh(X/G)$ such that $\Pi^{G,*}(\sV)$ is isomorphic to $\sF$.
\end{defn}

By using the adjunction properties of the functors involved, Definition \ref{def:descendcoherent} turns out to be equivalent to requiring $\Pi^{G,*} \circ \Pi_{*}^G$ to be isomorphic to the identity functor of $\coh^G(X)$, where $\Pi^*_G$ denotes $\pi^* \circ \triv$ and $\Pi_*^G$ denotes $[\cdot]^G \circ \pi_*$.\ We are interested in a similar property in the case of the subcategory of perfect complexes.

\begin{defn}\label{3.2.3}
Let $\sF^\bullet$ be a $G$-equivariant complex of sheaves in $\cD^G(X)$. We say that $\sF^\bullet$ \textit{descends} to $\perf(X/G)$ if there exists a complex $\sV^\bullet$ in $\perf(X/G) \subset \cD(X/G)$ such that  $\Pi^{G,*}(\sV^\bullet)$ is isomorphic to $\sF^\bullet$.
\end{defn}
As it was for sheaves, it follows from the definition that if $\sF^\bullet$ descends to $\perf(X/G)$, then 
$$\Pi^{G,*} \circ \Pi ^G_*(\sF^\bullet) \simeq \sF^\bullet,$$ 
namely the projection has a left inverse. If we work in a open neighborhood of a point it is just the identity map, and this implies that the descent property is local.

The theorem below gives us a necessary and sufficient condition for $\Omega$ to be an equivalence.

\begin{thm}\label{teoremadellasvolta}
The functor{\ } $\Omega$ is an equivalence between $\perf(X/G)$ and $\perf(Y/H)$ if and only if the two following conditions are satisfied:
\begin{itemize}
\item[(i)] The object $\Phi_{\sE} \circ \Pi^{G,*}(\sF^{\bullet})$ descends to $\perf(Y/H)$, for all $\sF^{\bullet} \in \perf(X/G)$.
\item[(ii)] The object $(\Phi_{\sE})^{-1} \circ \Pi^{H,*}(\sV^{\bullet})$ descends to $\perf(X/G)$, for all $\sV^{\bullet}\in \perf(Y/H)$.
\end{itemize}
\end{thm}

\begin{proof}
Let us first assume (i) and (ii) holds and prove that $\Omega$ is indeed an equivalence with image in $\perf(Y/H)$. By construction, the functor $\Omega$ is defined from the category $\perf(X/G)$ to $\cD(Y/H)$. However, condition (i) guarantees us that every object $\Omega(\sF^\bullet)$ is isomorphic to an object in $\perf(Y/H)$.
In order to prove that $\Omega$ is fully faithful we recall that both $\Pi^{G,*}$ and $\Phi_\sE$ are fully faithful.\ Moreover, the descent in Condition (i) is equivalent to asking that $\Pi^{H,*} \circ \Pi^{H}_{*}$ restricted to the image of  $\Phi_{\sE} \circ \Pi^{G,*}$ is isomorphic to the identity $\mbox{Id}|_{\Phi_{\sE}\circ \Pi^{G,*}}$, hence $\Pi^H_*$ is fully faithful as well.\ Hence, the functor $\Omega$ is fully faithful because it is composition of three fully faithful functors.
It remains to check that $\Omega$ is essentially surjective. Since $\Phi_{\sE^\bullet}$ is an equivalence, we can consider the inverse $(\Phi_{\sE^\bullet})^{-1}$. Asking that $\Omega$ is essentially surjective is the same as asking that the functor $\Pi_*^G \circ (\Phi_{\sE^\bullet})^{-1} \circ \Pi^{H,*}$ is fully faithful. Condition (ii) guarantees exactly that.

Vice versa, assume that $\Omega$ is an equivalence between $\perf(X/G)$ and $\perf(Y/H)$. If we take an object $\sF^{\bullet} \in \perf(X/G)$, then $\Omega(\sF^\bullet)$ belongs to $\perf(Y/H)$ and it follows by definition that $\Phi_{\sE^\bullet} \circ \Pi^{G,*}(\sF^{\bullet})$ descends to $\Omega(\sF^\bullet)$. It remains to prove that Conditions (ii) also holds: since $\Omega$ is an equivalence we can consider its inverse $\Omega^{-1}$. Notice that it is given by $\Pi^{G}_{*}\circ \Phi_{\sE^\bullet}^{-1} \circ \Pi^{H,*}$ restricted to $\perf(Y/H)$. By applying the same reasoning as before take an object $\sV^{\bullet}\in \perf(Y/H)$, then $\Omega^{-1}(\sV^{\bullet})$ belongs to $\perf(X/G)$ and it follows by definition that $(\Phi_{\sE^\bullet})^{-1} \circ \Pi^{H,*}(\sV^{\bullet})$ descends to $\Omega^{-1}(\sV^{\bullet})$.
\end{proof}

Notice that actually Condition (i) is used in order to provide the fact that $\Omega$ has the right image and that is fully faithful. Condition (ii) is needed for the essential surjectivity of $\Omega$ and it can be interpreted, when it makes sense, as Condition (i) rephrased for the functor $\Omega^{-1}:=\Pi^{G}_{*}\circ \Phi_{\sE}^{-1} \circ\Pi^{H,*}$. 

\subsection{Positive and negative examples} \label{sec:positivenegativeexamples}

Let us now focus on some examples, in order to understand better the behaviour of this descent criterion. For simplicity, we will now set $X=Y$ and $G=H$. We give an example which shows how to build an autoequivalence of $\perf(X/G)$ starting from a Fourier-Mukai autoequivalence of $\cD^G(X)$. In particular, we deal with the case of projective varieties with ample canonical or anticanonical bundle, since all the autoequivalences of $\cD^G(X)$ are classified:

\begin{thm}\emph{\cite[Theorem 7.2.]{Kaw}}
Let $X$ be a smooth normal projective variety with ample canonical sheaf or ample anticanonical sheaf with the action of a finite group $G$. Assume that the canonical sheaf generates the local divisor class group at each point of $X$. Then, the group of isomorphism classes of exact autoequivalence $\cD^G(X)$ is generated by shifts, tensor products with $G$-equivariant invertible sheaves and push forward along $G$-equivariant automorphisms of $X$.
\end{thm}

In order to study weather an autoequivalence $\Phi_{\sE^\bullet}$ of $\cD^G(X)$ induces and autoequivalence of $\perf(X/G)$, it suffices to study what happens when $\Phi_{\sE^\bullet}$ is one of these generators. \medskip

The case of $\Phi_{\sE^\bullet}$ being a shift is straigthforward. Explicitly, ${\sE^\bullet}=\Delta^G_{*}(\sO_X[d])$ for a certain integer $d$, where we are taking the push forward of the map $\Delta:X \to X\times X$ composed with the functor $\triv$. Then, it is immediate to see that the corresponding $\Omega$ descends to an autoequivalence of $\perf(X/G)$. \medskip

Now assume that $\Phi_{\sE^\bullet}$ is a tensor product with an invertible $G$-equivariant sheaf $\sL$ on $X$, that is ${\sE^\bullet}=\Delta^G_{*}(\sL)$, in the same notation as before. In the context of vector bundles we have the following result:

\begin{thm}[Thm 2.3, \cite{drena}] \label{3.2.4}
A $G$-equivariant vector bundle $\sF$ on $X$ descends to $X/G$ if and only if the stabiliser ${G}_x$ acts trivially on the fibre $\sF \otimes k(x)$ for every $x$ in $X$.
\end{thm}

\begin{cor}\label{3.2.5}
Let $\sL$ be a $G$-equivariant line bundle on $X$.\ If $n$ is a multiple of the order of $G$, then $\sL^{\otimes n}$ descends to the quotient.
\end{cor}
\begin{proof}
Consider any point $x$ in $X$. The stabiliser $G_x$ is finite and its action on the fibres is represented by a one-dimensional homomorphism whose values must be roots of unity. After taking tensor powers for any multiple of the order of the group, the action of $G_x$ on the fibres of $\sL$ becomes trivial. The result follows by applying Theorem \ref{3.2.4}.
\end{proof}

As a consequence of the previous corollary the Fourier-Mukai transform associated to the kernel $\sE^\bullet:=\Delta_{*}(\sL)$ gives rise to a functor $\Omega:=\Pi^G_{*} \circ \Phi_{\sE^\bullet} \circ \Pi^{G,*}$:
\begin{align*}
\Omega:\perf(X/G)&\to\cD(X/G)\\
\sA^{\bullet}& \mapsto \Pi^G_{*}(\sL\otimes\Pi^{G,*}(\sA^{\bullet})).
\end{align*}
If $\sL$ descends to the quotient, the stabiliser $G_x$ acts trivially on $\sL\otimes k(x)$ for every $x$ in $X$, then $\Omega$ defines an autoequivalence of $\perf(X/G)$. If $\sL$ does not descends, then, according to Corollary \ref{3.2.5}, it suffices to take the tensor product of $\sL$ to the order of the group $G$. \medskip

Lastly, we deal with $G$-equivariant automorphisms, namely we assume ${\sE^\bullet}$ to be of the form $\sO_{\Gamma_f}$, where $\Gamma_f$ denotes the graph of a $G-$automorphism $f$ of $X$. Then $\Omega:=\Pi^G_{*} \circ \Phi_{\sE^\bullet} \circ \Pi^{G,*}$:
\begin{align*}
\Omega:\perf(X/G)&\to\cD(X/G)\\
\sA^{\bullet}& \mapsto \Pi^G_{*}(f_{*}(\Pi^{G,*}(\sA^{\bullet}))).
\end{align*}
 
Using the fact that the automorphisms of $X$ preserve the descending property, plus the fact that $\Pi^{G,*}(\sA^{\bullet})$ descends, by Theorem \ref{teoremadellasvolta}, we have that $\Omega$ descends to an equivalence of $\perf(X/G)$.
\medskip

Let us now consider an example of a functor $\Omega$ which does not descend. Consider the group $\bZ_2$ acting on $\bP^2$ by multiplying the last coordinate by $-1$. This particular action has already been described in the preliminaries, see Examples \ref{exa:action1} and \ref{exa:action2}. 
The $\bZ_2$-equivariant derived category $\cD^{\bZ_2}(\bP^2)$ has a nice description: it may be seen as the category of bounded complexes of free $\bZ_2$-$k[x,y]$-modules of finite type with generators having bounded degree, up to homotopy equivalence. See \cite{ter} and \cite{terArx} for more details. Moreover, we have an explicit way to check whether some objects are perfect or not. Indeed, it was shown in \cite{delorme} that the sheaves $\sO_{\bP^2(1,1,2)}(d)$ belongs to $\perf(\bP^2(1,1,2))$ if and only if $d$ is an even number.

Consider a Fourier-Mukai autoequivalence 
$$\Phi_{\Delta^{\bZ_2}_* \sO(1)}: \cD^{\bZ_2}(\bP^2) \to \cD^{\bZ_2}(\bP^2),$$
obtained by taking the equivariant push forward of the diagonal embedding $\Delta:\bP^2 \hookrightarrow \bP^2 \times \bP^2$, applied to $\sO(1)$.
It can be explicitly computed that, for every $\sA^\bullet$ in $\cD^{\bZ_2}(\bP^2)$, we have $\Phi_{\Delta^{\bZ_2}_* \sO(1)}(\sA^\bullet)=\sA^\bullet \otimes \sO(1)$.

The functor $\Omega$ sends the perfect object $\sO_{\bP(1,1,2)}$ to $\sO_{\bP(1,1,2)}(1)$, which is not perfect. Hence $\Omega$ cannot be an autoequivalence of $\perf(\bP(1,1,2))$.

By following the same lines, it is interesting to see what happens to the autoequivalence $\Phi_{\Delta^{\bZ_2}_* \sO(d)}$, for $d$ an even integer. In fact, if we take any $\sA^\bullet$ in $\perf(\bP[1,1,2])$, we have
$$\Omega(\sA^\bullet) = \Pi_*^{\bZ_2}(\sO_{\bP^2}(d) \otimes \pi^* \sA^\bullet) = \Pi_*^{\bZ_2}(\pi^* \sO_{\bP^2(1,1,2)}(d) \otimes \pi^* \sA^\bullet) = \sO_{\bP^2(1,1,2)}(d) \otimes \sA^\bullet.$$
This is enough to prove that, in this case, $\Omega$ descends.

\section{Criterion for the descent to perfect complexes} \label{sec:descent}
The aim of this section is to provide a criterion (Theorem \ref{PG}) for characterizing the complexes of $\cD^G(X)$ descending to perfect complexes of $\cD(X/G)$. 

We now want to exploit Theorem \ref{3.2.4} in order to get a descent criterion for $\perf(X/G)$. We start with two lemmas which give us homological informations about the abelian category $\coh^G(X)$: 

\begin{lem}\emph{\cite[Lemma 1.1.]{terArx}}\label{3.2.6}\
The homological dimension of $\coh^G(X)$ is at most $\dim(X)$. That is, for every pair of $G$-equivariant coherent sheaves $\sF$ and $\sG$, $\ext^i(\sF,\sG)=0$ for every $i \geq \dim(X)$.
\end{lem}

\begin{lem}\emph{\cite[Lemma 1.2.]{terArx}}\label{3.2.7}\
Every $G$-equivariant coherent sheaf admits a finite resolution of $G$-equivariant locally free sheaves of finite type.
\end{lem}

Given an element $\sF^\bullet$ in $\cD^G(X)$ it is not automatic that every sheaf composing $\sF^\bullet$ is $G$-equivariant. We will call \textit{finite locally free $G$-resolution} of $\sF^\bullet$ a complex $\sG^\bullet$ in $\cD^G(X)$ quasi-isomorphic to $\sF^\bullet$ such that, for every integer $i$, $\sG^{i}$ is a $G$-equivariant locally free sheaf of finite type.
The following corollary is an immediate consequence of lemmas \ref{3.2.6} and \ref{3.2.7}.

\begin{cor}\label{3.2.8}
Any $\sF^\bullet$ in $\cD^G(X)$ admits a finite locally free $G$-resolution.
\end{cor}

Now we are ready to give a descent criterion at the level of derived categories:

\begin{prp}\label{3.2.9}
The complex $\sF^\bullet$ in $\cD^G(X)$ descends to $\perf(X/G)$ if and only if there exists a finite locally free $G$-resolution $\sE^\bullet$ of $\sF^\bullet$ such that, for every $x$ in $X$ and every integer $i$, the stabiliser $G_x$ acts trivially on the fibre $\sE^{i} \otimes k(x)$.
\end{prp}
\begin{proof}
Assume that $\sF^\bullet$ descends to $\perf(X/G)$, i.e. there exists a complex $\sB^\bullet$ in $\perf(X/G)$ such that $\Pi^{G,*}(\sB^\bullet)$ is quasi-isomorphic to $\sF^\bullet$. 

By adjunction, we get that $\Pi^G_* \circ \Pi^{G,*}(\sF^\bullet)$ is isomorphic to $\sF^\bullet$. Hence we have that $\sB^\bullet=\Pi^G_*(\sF^\bullet)$. 
Take a complex of vector bundles on $X/G$ quasi-isomorphic to $\sB^\bullet$ and denote it explicitly by 
$$\sV^\bullet=\{ 0\to \sV^1 \to \ldots \to \sV^n \to 0 \}.$$ 

We can apply the map $\Pi^{G,*}$ on the complex $\sB^\bullet$ by applying the corresponding map of sheaves termwise. We have the following quasi-isomorphisms:
$$\sF^{\bullet} \cong \Pi^{G,*}(\sB^\bullet) \cong \Pi^{G,*}(\sV^\bullet)  \cong \{ 0\to \Pi^{G,*}(\sV^1)\to \ldots \to \Pi^{G,*}(\sV^n)\to 0 \}.$$

The property of being locally free sheaves is preserved by $\Pi^{G,*}$, hence $\Pi^{G,*}(\sV^\bullet)$ is a complex of vector bundles which is quasi-isomorphic to $\sF^\bullet$. Furthermore, for any point $x$ in $X$, the stabiliser $G_x$ acts trivially on the fibres of $\Pi^{G,*}(\sV^i)$ since $\Pi^G_{*}(\Pi^{G,*}(\sV^i))=\sV^i$ is a vector bundle on $X/G$. This holds for every $i=1, \ldots, n$, proving the claim.\medskip

Vice versa, assume that there exists a finite locally free $G$-resolution $\sE^\bullet$ of $\sF^\bullet$ such that, for every $x$ in $X$, the stabiliser $G_x$ acts trivially on the fibre $\sE^i \otimes k(x)$ for every integer $i$.

The functor $\Pi^G_*(\sE^\bullet)$ is exact and, thanks to Theorem \ref{3.2.4}, we have that every $\sE^i \otimes k(x)$ descends to a vector bundle on the quotient. It follows that the image $\Pi^G_*(\sE^{\bullet})$ is a complex of vector bundles on $X/G$, that is an element of $\perf(X/G)$.
\end{proof}

Now we want to simplify this criterion by showing that it is possible to check the hypothesis concerning the stabiliser $G_x$ directly on the cohomology of the complex $\sF^\bullet$ and not on a specific finite locally free $G$-resolution. This is the content of Theorem \ref{PG}. We recall the following result about the descent of coherent sheaves.

\begin{lem}\emph{\cite[Lemma 2.14.]{nev}}\label{3.2.10}
Let $Z$ be an affine scheme with an action of a finite group $G$.\ Let $\sF$ be a $G$-equivariant coherent sheaf on $Z$, such that the action of the stabiliser $G_z$ on the vector space $\sF\otimes k(z)$ is trivial for every $z$ in $Z$. Then, there exists a $G$-equivariant locally free sheaf $\sE$ on $Z$ with a surjective map $\sE\to\sF$ such that $\sE$ descends to the quotient $Z/G$.
\end{lem}

The following result is an adaptation to our case of the replacement tool found in \cite[Proposition 4.1.]{nev}.

\begin{lem}\label{3.2.12}
Consider a finite locally free $G$-resolution of $\sF^\bullet$ in $\cD^G(X)$:
$$
\xymatrix{
\sE^\bullet:=\{0\ar[r]&\sE^1\ar[r]^{\alpha_1} &\sE^2\ar[r]^{\alpha_2}&\ldots \ar[r]^-{\alpha_{n-2}}&\sE^{n-1}\ar[r]^-{\alpha_{n-1}}&\sE^n\ar[r]&0\}.
}
$$
If the action of the stabiliser $G_x$ is trivial on the $\sO_X$-modules $H^n(\sE^\bullet \otimes k(x))$, for every point $x$ in $X$, then $\sF^\bullet$ is quasi-isomorphic to a complex of $G$-equivariant locally free sheaves
$$
\xymatrix{
\sV^\bullet:=\{0\ar[r]&\sV^{-m}\ar[r] &\ldots \ar[r]&\sV^{0}\ar[r]&\sV^{1}\ar[r]& \ldots \ar[r]&\sV^{n}\ar[r] & 0\}
}
$$
such that $\sV^n$ descends to the quotient.
\end{lem}
\begin{proof}
Notice that we will get the descent property only for the sheaf $\sV^n$, this is enough for the proof of the main result. We will start by constructing explicitly the $\sV^i$, $i\geq 0$, while the $i < 0$ part will be constructed as a resolution of sheaves and glued to the previous one. The descent property is local, hence we can assume that $X$ is affine. We have a surjection $\rho_{n}:\sE^n\to \coker(\alpha_{n-1})$; this induces another surjection for all $x\in X$, given by tensoring with $k(x)$: 
$$\sE^n \otimes k(x)\to \coker (\alpha_{n-1}) \otimes k(x).$$
We know that $H^n(\sF^\bullet \otimes k(x)) = H^n(\sE^\bullet \otimes k(x))=\coker(\alpha_{n-1})\otimes k(x)$ is $G_x$-invariant. We can apply Lemma \ref{3.2.10}, which implies the existence of a $G$-equivariant locally free sheaf $\sV^n$ with a surjective morphism $\gamma_n:\sV^n\to \coker(\alpha_{n-1})$ and a morphism $f_n:\sV^n\to\sE^n$ such that $\rho_n \circ f_n=\gamma_n$. We found the last element of the complex $\sV^\bullet$.

In order to define $\sV^{n-1}$, consider the subsheaf $\sG^n\subset \sE^{n-1}\oplus \sV^n$ made by those sections $(e,v)$ such that $\alpha_{n-1}(e)=f_n(v)$. By using the second projection from the direct sum, we get a morphism $\beta'_{n-1}:\sG^n\to\sV^n$. 
Now we take a $G$-equivariant locally free sheaf $\sV^{n-1}$ with a surjective map onto $\sG^n$. By projecting on to the first summand, we obtain a map $f_{n-1}:\sV^{n-1}\to\sE^{n-1}$, and composing with $\beta'_{n-1}$ we obtain a morphism $\beta_{n-1}:\sV^{n-1}\to\sV^n$:
\[
\xymatrix{
\ldots\ar[r]&\mathscr{E}^{n-1}\ar[r]^{\alpha_{n-1}}&\sE^n\ar[r]\ar[dr]^{\rho_n}&0\\
&&& \coker(\alpha_{n-1})\\
&\sG^{n}\ar[uu]\ar[r]^{\beta'_{n-1}}&\sV^n\ar[uu]_{f_n}\ar[r]\ar[ur]_{\gamma_n}&0\\
&\ar@/^2pc/[uuu]^{f_{n-1}}\sV^{n-1}\ar[u]\ar[ur]_{\beta_{n-1}}&
}
\]
The map $f_n$ induces a map from $\sV^n/\im(\beta_{n-1})$ to $\coker(\alpha_{n-1})$ which is an isomorphism. It takes a standard diagram chasing argument to prove that $\ker(\gamma_n)=\im(\beta_{n-1}$): if a section $v$ of $\sV^n$ is such that $\gamma^n(v)=0$, then there exists a section $e$ of $\sE^{n-1}$ such that $\alpha_{n-1}(e)=f_n(v)$, and thus $\beta'_{n-1}(e,v) = v$; now it is sufficient to take a section $v'$ of $\sV^{n-1}$ such that $\beta'_{n-1}(v') = e$. If $v=\beta_{n-1}(v')=\beta_{n-1}(e, v)$ then  $\gamma_n(v)=\rho_n(\alpha_{n-1}(e))=0$. A similar argument allows us to prove that the map $\gamma_{n-1}:\ker(\beta_{n-1})\to \ker(\alpha_{n-1})/\im(\alpha_{n-2})$ is surjective as well. \medskip

We now iterate the process: consider the subsheaf $\sG^{n-1}\subset\sE^{n-2} \oplus\sV^{n-1}$ made by those sections $(e,v)$ such that $\alpha_{n-2}(e)=f_{n-1}(v)$. We have a natural morphism $\beta'_{n-2}:\sG^{n-1}\to \sV^{n-1}$. Take a $G$-equivariant locally free sheaf $\sV^{n-2}$ with a surjective map onto $\sG^{n-1}$. Then, we obtain a map $f_{n-2}:\sV^{n-2}\to\sE^{n-2}$, and composing with $\beta'_{n-2}$ we obtain a morphism $\beta_{n-2}:\sV^{n-2}\to\sV^{n-1}$. Finally, we get that $\ker(\gamma_{n-1})$ corresponds to $\im \beta_{n-2}$.

We have the following situation:
\[
\xymatrix{
\ldots\ar[r]&\sE^{n-2}\ar[r]^{\alpha_{n-2}}&\sE^{n-1}\ar[r]^{\alpha_{n-1}}&\sE^n\ar[r]&0\\
&&\sG^{n}\ar[u]\ar[r]^{\beta'_{n-1}}&\sV^n\ar[u]^{f_n}\ar[r]&0\\
&\sG^{n-1}\ar[uu]\ar[r]^{\beta'_{n-2}}&\ar@/^2pc/[uu]^{f_{n-1}}\sV^{n-1}\ar[u]\ar[ur]_{\beta_{n-1}}&\\
&\sV^{n-2}\ar@/^3pc/[uuu]^{f_{n-2}}\ar[u]\ar[ur]_{\beta_{n-2}}&&
}
\]
We can iterate this construction until we get a complex of $G$-equivariant vector bundles
\[
\xymatrix{
0\ar[r]&\sV^1 \ar[r]^{\beta_1}&\sV^2 \ar[r]^-{\beta_2}& \ldots \ar[r]^-{\beta_{n-2}}&\sV^{n-1}\ar[r]^{\beta_{n-1}}&\sV^n\ar[r]&0
}
\]
such that $\gamma_j:\sV^j \to \coker(\alpha_{j-1})$ induces an isomorphism $H^j(\sV^{\bullet})\simeq H^j(\sE^{\bullet})$, for $2 \le j\le n$, and $\gamma_1:\sV^1 \to \sE^1$ is surjective map. Moreover, by construction, $\sV^n$ descends to the quotient. Thus, we can take the resolution of the subsheaf $i: \ker(\gamma_1) \hookrightarrow \sV^1$:
\[
\xymatrix{
0\ar[r]&\mathscr{W}^{-m}\ar[r]& \ldots \ar[r]&\mathscr{W}^{-1}\ar[r]&\mathscr{W}^0\ar[r]^-{\omega}&\ker(\gamma^1)\ar[r]&0.
}
\]

We obtain in this way the final form of the complex $\sV^\bullet$.
\[
\xymatrix@C=2em{
0\ar[r]&\sW^{-m}\ar[r]& \ldots \ar[r]&\sW^{0}\ar[r]^{i\circ\omega}&\sV^1\ar[r]^{\beta_1}&\sV^2\ar[r]& \ldots \ar[r]&\sV^n\ar[r]&0
}.
\]
It is a complex of $G$-equivariant vector bundles quasi-isomorphic to $\sE^{\bullet}$ such that $\sV^n$ descends to the quotient. These sheaves glue together because they are all quasi-isomorphic to $\sF^\bullet$.
\end{proof}

We point out that since the complex $\sV^\bullet$ found in the previous lemma is quasi-isomorphic to $\sE^\bullet$, we have that $H^i(\sV^\bullet)=0$, for $i= -m, \ldots, 0$. We will need the following result about the action of a finite group acting on $\bC$-vector spaces.

\begin{lem}\label{3.2.11}
Let $H$ be a finite group acting on three $k$-vector spaces $V_1$, $V_2$ and $V_3$, sitting in an exact sequence
\[
\xymatrix{
V_1\ar[r]^{\alpha}&V_2\ar[r]^{\beta}&V_3,
}
\]
where $\alpha$ and $\beta$ are equivariant maps.
If $H$ acts trivially on $V_1$ and on $V_3$, then $H$ acts trivially on $V_2$ as well.
\end{lem}

\begin{proof}
Let $x$ be in $V_2$ and let $h$ be in $H$. If $\beta(x)=0$ then $x=\alpha(y)$ for a certain $y\in V_1$; so $hx=h\alpha(y)=\alpha(hy)=\alpha(y)=x$. Suppose now that $\beta(x)\neq 0$. We have that $\beta(hx)=h\beta(x)=\beta(x)$, so $\beta(hx-x)=0$. Hence, there exist $y \in V_1$ such that $x=hx-\alpha(y)$. Let make $h$ act again, since the action on $V_1$ is trivial we get $hx=h^2x-\alpha(y)$, which gives $x=h^2x-2\alpha(y)$.
Since $H$ is finite there exist an integer $n$ such that $h^n=\id$. By repeating the previous computation we get $x=h^nx-n\alpha(y)$, so $\alpha(y)=0$. This concludes the proof.
\end{proof}

Now we are ready to prove the main theorem of this section.

\begin{thm}\label{PG}
Let $G$ be a finite group acting on a smooth projective variety $X$.\ An object $\sF^\bullet$ in $\cD^G(X)$ descends to $\perf(X/G)$ if and only if the following condition holds.

\begin{itemize}
\item[{$(\star)$}] The stabiliser $G_x$ acts trivially on the $\sO_X$-modules $H^j(\sF^\bullet \otimes k(x))$, for all $x \in X, j \in \bZ$.
\end{itemize}
\end{thm}
\begin{proof}
Start by assuming that $\sF^{\bullet}$ descends, and prove that condition $(\star)$ holds. From the descent condition we get a finite $G$-resolution $\sE^{\bullet}$ of $\sF^{\bullet}$ obtained by Proposition \ref{3.2.9}. Since $G_x$ acts trivially on $\sE^{\bullet}\otimes k(x)$, then it must act trivially also on $H^j(\sE^{\bullet}\otimes k(x))$, for all $j$, because the action of the group commutes with taking cohomology. The claim follows from this isomorphism of $\sO_X$-modules:
\[
\xymatrix{
H^j(\sF^{\bullet}\otimes k(x))\ar[r]^{\sim}&H^j(\sE^{\bullet}\otimes k(x)).
}
\]

Vice versa, we assume that condition $(\star)$ holds.\ We will proceed by induction on the number $n$ of non-trivial cohomologies of $\sF^\bullet$, assuming that $H^j(\sF^{\bullet}) = 0$ unless $j = j_1, \ldots, j_n$.
If $n=1$ the statement reduces to the case in which we assume that $\sF^{\bullet}$ is a coherent sheaf $\sF$. The descent property is local, so we can suppose $X$ to be affine. If condition $(\star)$ holds, then in particular $G_x$ acts trivially on $\sF \otimes k(x)$ for every $x$ in $X$:
$$H^j(\sE^{\bullet} \otimes k(x))=H^0(\sF \otimes k(x))=\sF \otimes k(x).$$
By Lemma \ref{3.2.10}, there exists a $G$-equivariant locally free sheaf $\sV_1$ which descends to the quotient, and a surjective morphism $\phi:\sV_1\to\sF$. Now, we denote by $\sK_1$ be the kernel of morphism $\phi$ and we consider the following exact sequence:
\[
\xymatrix{
0\ar[r]&\sK_1\ar[r]&\sV_1\ar@^{->>}[r]^{\phi}&\sF\ar[r]&0.
}
\]
We take the long exact cohomology sequence:
\begin{align*}
\ldots &\to \tor_{j-1}(\sK_1,k(x)) \to 0 \to \tor_{j-1}(\sF,k(x))\to  \\
&\to \tor_{j-2}(\sK_1,k(x)) \to 0 \to \tor_{j-2}(\sF,k(x))\to \\
\ldots &\to \tor_{1}(\sK_1,k(x)) \to 0 \to  \tor_1(\sF,k(x))\to \\
&\to \sK_1 \otimes k(x) \to \sV_1\otimes k(x) \to \sF\otimes k(x) \to 0
\end{align*}

The vector space $H^j(\sK_1\otimes k(x))$ is always between two $G_x$-invariant vector spaces, with the maps being $G_x$-equivariant, so Lemma \ref{3.2.11} implies that it is $G_x$-invariant as well. Now we iterate the process; apply again Lemma \ref{3.2.10} to get a $G$-equivariant locally free sheaf $\mathscr{V}_2$ which descends to the quotient, and a surjective morphism $\phi':\mathscr{V}_2\to\mathscr{K}_1$. Let $\mathscr{K}_2$ be the kernel of $\phi'$. As before, we have the exact sequence:
\[
\xymatrix{
0\ar[r]&\sK_2\ar[r]&\sV_2\ar@^{->>}[r]^{\phi'}&\sK_1\ar[r]&0.
}
\]
From the long exact cohomology sequence associated to the sequence above, we deduce that $H^j(\sK_2\otimes k(x))$ is $G_x$-invariant. In particular there exists a $G$-equivariant locally free sheaf $\sV_3$ which descends to the quotient, and a surjective morphism from $\sV_3$ to $\sK_2$. Since the homological dimension of $\coh^G(X)$ is finite, the process must end in a finite number of iterations. We denote by $m$ this number. It means that, at some point we obtain a kernel $\sK_m$ which is a locally free sheaf and such that the action of the stabiliser $G_x$ on the vector spaces $\sK_m\otimes k(x)$ is trivial for every $x\in X$.\ Therefore, by Theorem \ref{3.2.4}, $\sK_m$ descends to $\perf(X/G)$. In other words we obtain a finite $G$-equivariant locally free resolution of $\sF$ of the form:
\[
\xymatrix{
0\ar[r]&\sV_m\ar[r]&\ldots \ar[r]&\sV_2\ar[r]&\sV_1\ar[r]&\sF\ar[r]&0.
}
\]
such that every $\mathscr{V}_i$ descends to the quotient for $i=1, \ldots, m$. \medskip

Assume now that the statement is true for $n-1$.
Consider the following resolution of $\sF^\bullet$:
$$\sE^{\bullet}:=\{0\to\sE^{1}\to \ldots \to\sE^{n}\to 0 \}.$$
By Lemma \ref{3.2.12} there exists a complex $\sV^{\bullet}$ which is quasi-isomorphic to $\sE^\bullet$ and $\sV^{n}$ descends to the quotient. So, we have the following distinguished triangle:
\[
\xymatrix@C=1.4pc{
\sV^n[-n]\otimes k(x)\ar[r]&\sV^{\bullet}\otimes k(x)\ar[r]&\sigma_{\le n-1}\sV^{\bullet}\otimes k(x)\ar[r]&\sV^n[-n]\otimes k(x)[1],
}
\]
where $\sigma_{\le n-1}\sV^{\bullet}$ is the stupid truncation of the complex $\sV^{\bullet}$. Consider the long exact cohomology sequence:
\begin{multline*}
0\to H^{n-1}(\sV^{\bullet}\otimes k(x)) \to H^{n-1}(\sigma_{\le n-1}\sV^{\bullet}\otimes k(x)) \to\\
\to \sV^n[-n]\otimes k(x)\to H^{n}(\sV^{\bullet}\otimes k(x)) \to 0
\end{multline*}
By using Lemma \ref{3.2.11} we have that $H^{n-1}(\sigma_{\le n-1}\sV^{\bullet}\otimes k(x))$ is $G_x$-invariant. Furthermore, 
$H^{j}(\sigma_{\le n-1}\sV^{\bullet}\otimes k(x))$ is isomorphic to $H^{j}(\sV^{\bullet}\otimes k(x))$ for $j=1, \ldots, n-1$. 
It means that $\sigma_{\le n-1}\sV^{\bullet}$ is a complex such that $H^{j}(\sigma_{\le n-1}\sV^{\bullet}\otimes k(x))$ is $G_x$ invariant for all $j$.\ Moreover, $H^{j}(\sigma_{\le n-1}\sV^{\bullet})\neq 0$ for $j=1, \ldots, n-1$. Hence we can apply the induction hypothesis to conclude that $\sigma_{\le n-1}\sV^{\bullet}$ descends to $\perf(X/G)$. We have the following commutative diagram:
\[
\xymatrix{
\sV^{n}[-n]\ar[r]\ar[d]&\sV^{\bullet}\ar[r]\ar[d]&\sigma_{\le n-1}\sV^{\bullet}\ar[d]\\
\Pi^{G,*}\Pi^G_{*}(\sV^{n}[-n])\ar[r]&\Pi^{G,*}\Pi^G_{*}(\sV^{\bullet})\ar[r]&\Pi^{G,*}\Pi^G_{*}(\sigma_{\le n-1}\sV^{\bullet})
}
\]
The first and third vertical map are quasi-isomorphism, by the five-Lemma this implies that the second vertical map is a quasi-isomorphism, it means that $\sV^{\bullet}$ descends to $\perf(X/G)$ and this concludes the proof.
\end{proof}

\section{Application of the criterion} \label{sec:applications}
In the previous section we found a criterion characterising the descending complexes. We can then make Theorem \ref{teoremadellasvolta} by using Theorem \ref{PG} more explicit.
Let $\Phi_{\sE^\bullet}$ be a Fourier-Mukai equivalence from $\cD^G(X)$ to $\cD^H(Y)$. We begin with a lemma which collects together the result of the previous sections.

\begin{lem}\label{PG2}
The functor { }$\Omega$ is an equivalence between $\perf(X/G)$ and $\perf(Y/H)$ if and only if the two following conditions are satisfied for any integer $i$, any $y$ in $Y$, any $x$ in $X$, any $\sA^\bullet$ in $\perf(X/G)$ and any $\sB^\bullet$ in $\perf(Y/H)$.
\begin{itemize}
\item[(i')] The stabiliser $H_y$ acts trivially on $\left[H^i\bigl(\Pi^{*,G} (\sA^\bullet) \otimes \sE^\bullet|_{X \times \{y\}} \bigr)\right]^G$.

\item[(ii')] The stabiliser $G_x$ acts trivially on $\left[H^i\bigl(\Pi^{*,H} (\sB^\bullet) \otimes \sE_L^\bullet|_{\{x\} \times Y} \bigr)\right]^H$.

\end{itemize}
Here $\sA^\bullet$ and $\sB^\bullet$ are interpreted as complexes on $X \times \{y\}$ and $\{x\} \times Y$, respectively.
\end{lem}

\begin{proof}
We want to apply Theorem \ref{PG} to the two conditions of Theorem \ref{teoremadellasvolta}. We will deal here only with Condition (i'), which comes from Condition (i); the second condition works in the same way. 

By Theorem \ref{PG}, condition (i) holds if and only if the stabiliser $H_y$ acts trivially on the $\sO_Y$-module $H^j(\Phi_{\mathscr{E}^{\bullet}}\circ \Pi^{G,*}(\sA^\bullet) \otimes k(y))$ for every integer $j$, for every $\sA^\bullet\in \perf(X/G)$ and $y \in Y$. We now want to rephrase this condition in order to highlight the role of the kernel $\sE^{\bullet}$. We can write it explicitly as
\begin{equation*} 
H_y \text{ acts trivially on } H^j\left[  q^G_{Y,*}\bigl(q_X^{*} \circ \Pi_X^{*,G} (\sA^\bullet) \otimes \sE^{\bullet} \bigr)    \otimes k(y)     \right],
\end{equation*}
again for every point $y$ of $Y$, every $\sA^{\bullet}\in \perf(X/G)$ and every integer $j$. Recall that $q^G_{Y,*}$ denotes the equivariant pushforward with respect to the kernel $G$ of the projection $G \times H \to H$. Let us focus just on the expression inside the cohomology, writing the equivariant part explicitly:
\begin{equation*} 
\left[  q_{Y,*}\bigl(q_X^{*} \circ \Pi_X^{*,G} (\sA^\bullet) \otimes \sE^{\bullet} \bigr)    \otimes k(y)     \right]^G.
\end{equation*}

The inclusion of $y$ in $Y$ defines a natural embedding $\iota$ of $X \times \{y\}$ in $X \times Y$. We have the following commutative diagram, where by an abuse of notation we denote the parallel arrows by the same name.

$$\xymatrix{
X \times \{y\}\ar@{^{(}->}[r]^{\iota_y}\ar[d]^{q_Y} & X \times Y\ar[d]^{q_Y}\\
\{y\}\ar@{^{(}->}[r]^{\iota_y} & Y
}$$

The group $G$ acts trivially on $Y$, the map $\iota_y$ can be used to bring $k(y)$ inside the map $[-]^G$, see \cite[Section 2.2]{Krug}. Then, by base change, inside the pushforward $q_{Y,*}$ as well:
\begin{equation*} 
\left[q_{Y,*} \circ \iota_y \bigl((q^*_X \circ \Pi^{*,G} (\sA^\bullet))\otimes \sE^\bullet\bigr)\right]^G.
\end{equation*}
The map $\iota_y$ is defined fibrewise and it is just the restriction to $X \times \{y\}$, so we get
\begin{equation*} 
\left[q_{Y,*} \bigl((q^*_X \circ \Pi^{*,G} (\sA^\bullet))|_{X \times \{y\}}\otimes \sE^\bullet|_{X \times \{y\}} \bigr)\right]^G.
\end{equation*}
Notice that $\iota_y$ and $q_X$ are inverse to each other, after the identification of $X$ with $X \times \{y\}$, so we have
\begin{equation*} 
\left[q_{Y,*} \bigl(\Pi^{*,G} (\sA^\bullet) \otimes \sE^\bullet|_{X \times \{y\}} \bigr)\right]^G,
\end{equation*}
where we remark that we are considering the complex $\Pi^{*,G} (\sA^\bullet)$ as a complex on $X \times \{y\}$.

The map $q_{Y,*}$ is the push forward to the point $y$, so it correspond to taking global sections
\begin{equation*} 
\left[ \Gamma\bigl(X \times \{y\}, \Pi^{*,G} (\sA^\bullet) \otimes \sE^\bullet|_{X \times \{y\}} \bigr)\right]^G,
\end{equation*}
Every complex of vector spaces splits, so we can write
\begin{equation*} 
\left[ \bigoplus_{j \in \bZ} H^j\bigl(\Pi^{*,G} (\sA^\bullet) \otimes \sE^\bullet|_{X \times \{y\}} \bigr)[-j]\right]^G.
\end{equation*}
Finally, notice that the direct sum and the shift commutes with taking $G$ invariants. Writing the whole condition as in the beginning we get
\begin{equation*} 
 H_y \text{ acts trivially on } H^i\left(\bigoplus_{j \in \bZ} [H^j\bigl(\Pi^{*,G} (\sA^\bullet) \otimes \sE^\bullet|_{X \times \{y\}} \bigr)]^G[-j]\right).
\end{equation*}

By computing explicitly the outer cohomology we arrive to the final result.
\end{proof}

We take a moment to observe that at the level of linearisations we get only the contribution coming from the kernel $\sE^\bullet$, since $\sA^\bullet$ and $\sB^\bullet$ are endowed with the trivial linearisation. Moreover, notice that the action of $H_y$ on $\left[H^i\bigl(\Pi^{*,G} (\sA^\bullet) \bigr)\right]^G$ is trivial, however we can not remove $\sA^\bullet$ from the condition due to the fact that the cohomology and the tensor product does not commute. This is summarised in the following result.

\begin{cor} \label{cornec}
The functor $\Omega$ is an equivalence between $\perf(X/G)$ and $\perf(Y/H)$ only if the two following conditions are satisfied for any integer $i$, $x$ in $X$ and $y$ in $Y$. 
\begin{itemize}
\item[(i'-nec)] The stabiliser $H_y$ acts trivially on $\left[H^i\bigl(\sE^\bullet|_{X \times \{y\}} \bigr)\right]^G$.

\item[(ii'-nec)] The stabiliser $G_x$ acts trivially on $\left[H^i\bigl(\sE_L^\bullet|_{\{x\} \times Y} \bigr)\right]^H$.
\end{itemize}
\end{cor}
\begin{proof}
This comes immediately by chosing as $\sA^\bullet$ the complex $\sO_{X/G}$ and as $\sB^\bullet$ the complex $\sO_{Y/H}$. Notice that there is no hope in general for these conditions to be sufficient as well: the kernel $\sE^\bullet$ may have everywhere vanishing cohomologies, so (i'-nec) and (ii'-nec) are empty, but (i') and (ii') are not.
\end{proof}

We can use this corollary to deal with the case proposed in Section \ref{sec:positivenegativeexamples}, for the kernel $\sE^\bullet$ being $\Delta_*^{\bZ_2}\sO(1)$. In (i'-nec), the only non trivial cases to check correspond to $i=0$ and $y=(x_0:x_1:0)$ or $(0:0:1)$. By explicit calculations (or by directly applying Theorem \ref{3.2.4}) we see that (i'-nec) is not satisfied.
We conclude this section with the following theorem, which gives a condition that should be feasible to check in concrete cases, provided that a generator of the categories $\perf(X/G)$ and $\perf(Y/H)$ is known. 

\begin{thm} \label{thm:final2}
Conditions (i') and (ii') of Lemma \ref{PG2} can be checked just for a generator $\sA^\bullet$ of $\perf(X/G)$ and a generator $\sB^\bullet$ of $\perf(Y/H)$, respectively.
\end{thm}
\begin{proof}
Notice first that both $\perf(X/G)$ and $\perf(Y/H)$ can be generated by just one object, as proved in \cite[Theorem 4]{Orl09}. Let us focus on Condition (i'). We want to show that $(i')$ is stable under shifts and cones. The first property is guaranteed by the following equality
$$H^i\bigl(\Pi^{*,G} (\sA^\bullet[k]) \otimes \sE^\bullet|_{X \times \{y\}} \bigr) = H^{i+k}\bigl(\Pi^{*,G} (\sA^\bullet) \otimes \sE^\bullet|_{X \times \{y\}} \bigr).$$
For the second one, assume that $\sA^\bullet_1$ and $\sA^\bullet_2$ satisfy (i') for every $i$ and $y$. Let $\sC^\bullet$ be the cone of any morphism $\sA^\bullet_1 \to \sA^\bullet_2$. Our goal is to prove that $(i')$ holds for $\sC^\bullet$. The maps in the distinguished triangle associated to $\sC^\bullet$ induce maps between vector spaces
\begin{multline*}
    \left[H^i\bigl(\Pi^{*,G} (\sA^\bullet_1) \otimes \sE^\bullet|_{X \times \{y\}} \bigr)\right]^G \to \left[H^i\bigl(\Pi^{*,G} (\sC^\bullet) \otimes \sE^\bullet|_{X \times \{y\}} \bigr)\right]^G \to \\
    \to \left[H^i\bigl(\Pi^{*,G} (\sA^\bullet_2) \otimes \sE^\bullet|_{X \times \{y\}} \bigr)\right]^G.
\end{multline*}
These maps are $H$-equivariant, since they come from maps in $\perf(X/G)$ and from the functor $\triv$, hence we can apply Lemma \ref{3.2.11}, which guarantees that $H_y$ acts trivially in the middle term.
\end{proof}

We remark that the kernel $\sE^{\bullet}_L \cong \sE^{\bullet}_R$ of the inverse of $\Phi_{\sE^\bullet}$ has an explicit description as in Section \ref{sec:background}. This easily allows to write a more explicit version of conditions (ii), (ii'), (ii'-nec) by using similar computations as in Lemma \ref{PG2}. However, we still need to verify two statements to prove that $\Omega$ is an equivalence, one concerning the source category $\perf(X/G)$ and one the target $\perf(Y/H)$.

\section*{Acknowledgements}
The first idea of the paper was developed in the Ph. D. thesis of the first author, \cite{pigi}.
R.M. is supported by the Department of Mathematics and Natural Sciences of the University of Stavanger in the framework of the grant 230986 of the Research Council of Norway, by Firb 2012 Moduli spaces and applications granted by Miur and by MIUR: Dipartimenti di Eccellenza Program (2018-2022) - Dept. of Math. Univ. of Pavia.
We thank Paolo Stellari for many useful suggestions and for his comments that made us improve this short article. We are very grateful to Andreas Hochenegger and David Ploog for their valuable comments on a preliminary version of this paper. We finally thank an anonymous referee for pointing out a mistake in the first definition of the functor $\Omega$, which helped us correcting and improving the main result of the paper.

\bibliographystyle{alpha}
\bibliography{bibliografia}

\end{document}